\documentclass[1p]{elsarticle}

\usepackage{lineno}
\usepackage[colorlinks=true,linkcolor=darkblue,pdfborder={0 0 0}]{hyperref}

\journal{arxiv}

\modulolinenumbers[10]

\usepackage[latin9]{inputenc}

\usepackage[OT1, T1]{fontenc}

\usepackage{amsmath}
\usepackage{amsthm}
\usepackage{amssymb}

\newcommand{\totphi}{\varphi}%
\newcommand{\id}{\text{id}}%
\newcommand{\mygcd}{\gcd}%
\newcommand{\mytheta}[3][]{\theta_{{#2},{#3}}}%
\newcommand{\mymin}[3][]{\min({#2},{#3})}%

\newcommand{\sz}{s^\prime}

\theoremstyle{plain}
\newtheorem{theorem}{Theorem}
\newtheorem{lemma}{Lemma}
\newtheorem{proposition}{Proposition}
\newtheorem{corollary}{Corollary}

\theoremstyle{definition}

\theoremstyle{remark}

\newtheorem{case}{Case}
\newtheorem{step}{Step}

\begin{document}

\begin{frontmatter}


\title{Prime number decomposition of the Fourier transform of a function of the greatest common divisor.}

\author{L.J. Holleboom\fnref{myfootnote}}
\address{Department of Mathematics and Computer Science}
\address{Karlstad            University\linebreak
Universitetsgatan 2\linebreak 651 88 Karlstad\linebreak Sweden}

\fntext[myfootnote]{Corresponding author, email: thijs.hollebooom@kau.se}




\begin{abstract}

The discrete Fourier transform of the greatest common divisor is
a multiplicative  function, if  taken with  respect to  the same
order of  the primitive  root of  unity, which  is a  well known
fact.  As such,  the transform  can  be expressed  in the  prime
factors of the argument, the explicit form of which is proven in
this paper.  Subsequently it is  shown how the procedure  can be
generalized to the  discrete Fourier transform of  a function of
the  greatest  common  divisor. From  this  representation  some
interesting relations concerning the  Euler totient function and
generalizations thereof are established.

\end{abstract}


\end{frontmatter}


\section{Introduction}

The discrete  Fourier transform of the  greatest common divisor,
$\mygcd$, is commonly defined as \cite{Schramm2008}

\begin{equation}\label{eq000-1}
h_m(n) = \sum_{k=1}^n \gcd(k,n)e^{-k\frac{2\pi  i}{n}m}.
\end{equation}

The function $h_m(n)$ is multiplicative  in the argument $n$ for
fixed $m$  \cite{Schramm2008,vanderKamp2013}. Also,  $h_m(n)$ is
integer  valued.  This  follows  from  the  fact  that  $h_m(n)$
can  be written  as the  Dirichlet convolution  of the  identity
function and the Ramanujan sum, as proven in \cite{Schramm2008},
and  that the  Ramanujan sum  itself  is integer,  as proven  in
\cite{Kluyver1906}.  We come  back to  this later,  see equation
\ref{eq015} below.

Although   the   multiplicativity   of   $h_m(n)$   already   is
established, as described above, we outline a direct proof below
because this property is central  in the current work. Hence the
following proposition is formulated.

\begin{proposition}\label{proposition001}

The  discrete Fourier  transform  $h_m(n)$  defined in  equation
\ref{eq000-1} is  a multiplicative function in  the argument $n$
for a fixed value of the index $m$.
\end{proposition}

\begin{proof}[Sketch      of      proof      of      proposition
\ref{proposition001}]

Assume  two integer  numbers  $u$, and  $v$,  that are  coprime,
i.e.,  $\mygcd(u,v)  =  1$,  as  required  by  the  property  of
multiplicativity, then

\begin{eqnarray}
h_m(u)  h_m(v)
 &  = &
  \sum_{k=1}^u \mygcd(k,u)  e^{-m\frac{2\pi
i}{u}k} \sum_{l=1}^v \mygcd(l,v) e^{-m\frac{2\pi i}{v}l}
 \label{eq900-1}\\
 &  = &
\sum_{k=1}^u \sum_{l=1}^v \mygcd(kv + lu,uv) e^{-m\frac{2\pi i}{uv}(kv + lu)}
 \label{eq900-2}\\
&  = &
\sum_{a=1}^{uv} \mygcd(a,uv) e^{-m\frac{2\pi i}{uv}a}
 \label{eq900-3}\\
&  = &
h_m(uv)
 \label{eq900-4}
\end{eqnarray}

Here equation \ref{eq900-2}  follows from equation \ref{eq900-1}
because of elementary properties of the greatest common divisor:
$\mygcd(k,u) = \mygcd(kv,u) =  \mygcd(kv + lu,u)$, and similarly
$\mygcd(l,v) = \mygcd(kv + lu,  v)$, where the summation indices
$k$,  and $l$  are limited  to  the ranges  in their  respective
summation.  Subsequently, the  multiplicativity of  the greatest
common  divisor itself,  $\mygcd(kv  + lu,  u)  \mygcd(kv +  lu,
v)  =  \mygcd(kv +  lu,  uv)$  was used.  Furthermore,  equation
\ref{eq900-3} follows from equation \ref{eq900-2} because of the
Chinese remainder theorem  which, in this case, says  there is a
unique mapping between the numbers $k$ modulo $u$ and $l$ modulo
$v$, and the number $a$ modulo $uv$.

\end{proof}

It  is clear  that the  reasoning  in the  proof of  proposition
\ref{proposition001}  also holds  if one  replaces the  greatest
common  divisor by  a  multiplicative function  of the  greatest
common   divisor  everywhere   in  equations   \ref{eq900-1}  to
\ref{eq900-3}, i.e., replaces $\mygcd(k,u)$ by $f(\mygcd(k,u))$.

\section{Prime  number representation  of  the discrete  Fourier
transform of the greatest common divisor}

A       fundamental       property       of       multiplicative
functions\cite{Apostol1976}  is  that   the  function  value  is
entirely  expressible  in  function  values  of  powers  of  the
prime-factors of  the argument. The  explicit form for  the case
where this  function is  the discrete  Fourier transform  of the
greatest  common divisor,  i.e., $h_m(n)$,  is given  in theorem
\ref{th1} below. We define the symbol $\mytheta{t}{s} = H[t-s]$,
where  $H[w]$  is  the  Heaviside  step  function  for  integral
arguments  $w$. Hence,  $\mytheta{t}{s} =  0$  if $t  < s$,  and
$\mytheta{t}{s}  = 1$  if $t  \ge s$.  The comma-symbol  between
the  indices  is added  to  avoid  confusion  in case  where  an
index is  made up  of an  expression. We  also use  the function
$\mymin{r}{s}$ indicating the minimum  of the integral arguments
$r$ and $s$.

\

\

\

\

\begin{theorem}[
]\label{th1}

If the  factorization of  the number $n$,  the argument,  in its
prime  factors  $p_j, j  =  1,  \cdots, r$,  is  given  as $n  =
\prod_{j=1}^r p_j  ^{s_j}$, where $s_j$,  with $ s_j \ge  1$, is
the multiplicity of prime factor $p_j$, and the order $m$ of the
primitive root of unity, where $1 \le m \le n$, is written as $m
= u \prod_{i=1}^r p_i ^{t_i}$,  with $t_i \ge 0$, $\mygcd(u,p_i)
= 1$, and  $1 \le i \le r$, then  the discrete Fourier transform
of the greatest common divisor can be expressed in prime factors
as

\begin{equation}\label{eq001}
\sum_{k=1}^n \mygcd(k,n) e^{-k\frac{2\pi  i}{n}m} =
  \prod_{i = 1}^r \left[ (\mymin{t_i}{s_i} + 1) \totphi(p_i^{s_i}) +
 \mytheta{t_i}{s_i} p_i^{s_i - 1} \right]
\end{equation}

\noindent where $\totphi(n)$ is the Euler totient function.

\end{theorem}

The  multiplicity  $s_j$ of  prime  $p_j$  in $n$  is  naturally
larger  than  0,  but  $m$   is  not  necessarily  divisible  by
$p_j$. Therefore,  a multiplicity  $t_i$ can  be equal  to zero,
indicating that prime  factor $p_i$ not is present  in the prime
factorization of  $m$. Also, since  $m$ ranges from $1$  to $n$,
the multiplicity $t_i$ of prime $p_i$  in $m$ can be higher then
the multiplicity $s_i$ of the same prime in $n$.

Two different  proofs will  be given, an  inductive proof  and a
constructive proof. The inductive proof provides more insight in
the  nature  of the  problem.  The  constructive proof  will  be
generalized to  the case  of the  discrete Fourier  transform of
\textit{a function} of the greatest common divisor, further down
in this work. In  both cases it will be used  that $h_m(n)$ is a
multiplicative function, meaning that  $h_m(n) = \prod_{i = 1}^r
h_m(p_i^{s_i})$, implying  that theorem \ref{th1} only  needs to
be proven  for the case where  $n$ equals the power  of a single
prime factor, i.e., $n = p^s$. The index $i$ can then be omitted
because only one factor is considered.

\begin{proof}[Proof 1, inductive proof]

The  two cases  $\mygcd(m,n)  >  1$ and  $\mygcd(m,n)  = 1$  are
considered separately, starting with the case $\mygcd(m,n) > 1$.
Induction with respect to $s$ will be applied.

\begin{case}[$\mygcd(m,n)  >   1$]\label{case001}  
\begin{step}[Base step] When $s = 1$ it follows that $n = p$ and
then also $m  = p$, because that is the  only possible value for
$m$ where  $\mygcd(m,p) > 1$. Then  also $t = 1$,  and the right
hand side of equation \ref{eq001}  becomes $2 \totphi(p) + p^0 =
2p-1$. For the left hand side one obtains

\begin{equation}
\sum_{k=1}^p \mygcd(k,p) e^{-k\frac{2\pi  i}{p}p} = \sum_{k=1}^p \mygcd(k,p)
= (p-1) 1 + p = 2p-1
\end{equation}
showing that equation \ref{eq001} holds for $s = 1$.
\end{step}

\begin{step}[Induction  step] Assume  that equation  \ref{eq001}
holds for $s =  \sz$, then with $n=p^{\sz}$ and $m  = u p^t$ the
induction hypothesis reads

\begin{equation}\label{eq002}
\sum_{k=1}^{p^{\sz}} \mygcd(k,p^{\sz}) e^{-k\frac{2\pi  i}{p^{\sz}}up^t} =
   (\mymin{t}{\sz} + 1) \totphi(p^{\sz}) +
 \mytheta{t}{\sz }p^{\sz - 1} 
\end{equation} 

Now consider  the left  hand side  of equation  \ref{eq002} with
$\sz$ replaced by $\sz+1$, and partition the summation over $k =
1 \cdots p^{\sz+1}$  into two subsets, I and II,  where subset I
consistes of all multiples of $p$ and subset II of the remaining
values. Hence  subset I  $ =  \{ p,  2p, \cdots  p^{\sz}p\}$ and
subset II $\{1, 2, \cdots p-1,  p+1,\cdots 2p-1 , 2p + 1, \cdots
p^{\sz}p - 1\}$. Consider the partial summation over k in subset
I using $k = j p$, this gives

\begin{eqnarray}
\sum_{\substack{k=1\\k\in I}}^{p^{\sz + 1}} \mygcd(k,p^{\sz+1})
  e^{-k\frac{2\pi  i}{p^{\sz+1}}up^t} & = & 
\sum_{j=1}^{p^{\sz}} \mygcd(jp,p^{\sz+1}) e^{-jp\frac{2\pi  i}{p^{\sz+1}}up^t}\label{eq004-1} \\
 & = & p \sum_{j=1}^{p^{\sz}} \mygcd(j,p^{\sz}) e^{-j\frac{2\pi  i}{p^{\sz}}up^t} \label{eq004-2}\\
 & = & p \left[(\mymin{t}{\sz} + 1) \totphi(p^{\sz}) + \mytheta{t}{\sz}p^{\sz-1}\right]\label{eq004-3}\\
 & = & (\mymin{t}{\sz} + 1) \totphi(p^{\sz + 1}) + \mytheta{t}{\sz}p^{\sz}\label{eq004-4}
\end{eqnarray}

where  in  the  step  to  equation  \ref{eq004-2}  it  has  been
used  that  $\mygcd(jp,p^{\sz+1})=   p  \mygcd(j,p^{\sz}),  j  =
1  \cdots  p^{\sz}$,  and  equation  \ref{eq004-3}  is  obtained
with  the induction  hypothesis. Equation  \ref{eq004-4} follows
by  using  that  the  totient function  obeys  $p\totphi(p^s)  =
\totphi(p^{s+1})$ for $s\ge 1$.

Next consider  the partial  summation over  subset II,  which is
such  that $\gcd(k,p^{\sz+1})=1$.  Upon rewriting  the summation
over subset II  as the difference of the  complete summation and
the summation over subset I, one obtains

\begin{eqnarray}
\sum_{\substack{k=1\\k\in II}}^{p^{\sz + 1}} \mygcd(k,p^{\sz+1})
  e^{-k\frac{2\pi  i}{p^{\sz+1}}up^t}
 & = & 
\sum_{\substack{k=1\\k\in II}}^{p^{\sz + 1}} 
  e^{-k\frac{2\pi  i}{p^{\sz+1}}up^t} \label{eq006-1}\\
 & = & 
\sum_{k=1}^{p^{\sz + 1}} 
  e^{-k\frac{2\pi  i}{p^{\sz+1}}up^t}
 - \sum_{k\in I} 
  e^{-k\frac{2\pi  i}{p^{\sz+1}}up^t}\label{eq006-2}\\
 & = & 
p^{\sz+1}\mytheta{t}{\sz+1} - \mytheta{t}{\sz} p^{\sz}\label{eq006-3}
\end{eqnarray}

where it  has been used  that the  first summation on  the right
hand  side  of equation  \ref{eq006-2}  equals  zero when  $t  <
\sz+1$, by summing the geometric  series. On the other hand, the
summand equals $1$ if $t  \ge \sz+1$, resulting in the summation
being equal to the number of terms, in this case. The second sum
also  equals zero,  except when  $t \ge  \sz$. Here  the summand
equals $1$ if $t \ge  \sz$, because the summation index contains
a factor $p$, and the summation  results in the number of terms,
$p^{\sz}$, in this case.

Collecting terms from equations \ref{eq004-4} and \ref{eq006-3},
and cancelling the term $\mytheta{t}{\sz} p^{\sz }$, gives

\begin{eqnarray}\label{eq007}
&&\sum_{k=1}^{p^{\sz + 1}} \mygcd(k,p^{\sz + 1}) e^{-k\frac{2\pi  i}{p^{\sz+1}}up^t} \nonumber\\
&=&
 (\mymin{t}{\sz} + 1)\totphi(p^{\sz + 1}) +  p^{\sz+1}\mytheta{t}{\sz+1} \label{eq007-1}\\ 
&=&
 (\mymin{t}{\sz + 1} + 1)\totphi(p^{\sz + 1}) -\mytheta{t}{\sz+1}(p^{\sz+1}- p^{\sz}) +
 p^{\sz+1}\mytheta{t}{\sz+1} \label{eq007-2}\\ 
&=&
 (\mymin{t}{\sz + 1} + 1)\totphi(p^{\sz + 1}) +\mytheta{t}{\sz+1} p^{\sz} \label{eq007-3} 
\end{eqnarray}

In  the step  from equation  \ref{eq007-1} to  \ref{eq007-2} the
property $\mymin{t}{s}  = \mymin{t}{s+1} -  \mytheta{t}{s+1}$ of
the minimum-function has been used, as well as the specific form
$  \totphi(p^a) =  p^a -  p^{a -  1}$ of  the totient  function.
Equation  \ref{eq007-3}  completes  the  induction  step,  which
started with equation \ref{eq002}.

\end{step}

\end{case}

\begin{case}[$\mygcd(m,n) = 1$]\label{case002}

Since  $n=p^s$, with  $p$  a  prime number,  and  $s  > 0$,  the
requirement  $\mygcd(m,  p^s)   =  1$  implies  that   $t  =  0$
in  equation  \ref{eq001},  which  implies that  the  term  with
$\mytheta{t}{s}$ always  equals 0.  Also $\mymin{0}{s} =  0$ for
all $s > 0$. Hence it has to be proven that

\begin{equation}\label{eq010}
\sum_{k=1}^{p^s} \mygcd(k,p^s) e^{-k\frac{2\pi  i}{p^s}m} =
   \totphi(p^{s}) 
\end{equation}

which is done in the following two steps.


\begin{step}[Base step,  s =  1] Direct  evaluation of  the left
hand side of equation \ref{eq010}, with $s = 1$, gives

\begin{eqnarray}\label{eq011}
\sum_{k=1}^p \mygcd(k,p) e^{-k\frac{2\pi  i}{p}m}
& = &
\sum_{k=1}^p e^{-k\frac{2\pi  i}{p}m} - \sum_{k=p}^p e^{-k\frac{2\pi  i}{p}m}
+\sum_{k=p}^p  p e^{-k\frac{2\pi  i}{p}m} \\
& = &
0 - 1 + p \\
& = &
\totphi(p)
\end{eqnarray}

which proves the base step.

\end{step} 

\begin{step}[Induction step]

The equation to be proven, equation \ref{eq010}, is identical to
equation \ref{eq002} with $t = 0$ and $\mytheta{t}{\sz}=0$, thus
we  can reuse  the  strategy used  there.  Assume that  equation
\ref{eq010}  holds for  $s =  \sz$, and  consider the  left hand
side  of equation  \ref{eq010} for  $s =  \sz +  1$. Upon  again
partitioning the summation  into the same two subsets  I and II,
as above,  and repeating  steps \ref{eq004-1}  to \ref{eq006-3},
where now all terms containing $\theta$-symbols or min-functions
vanish, one obtains

\begin{equation}\label{eq012}
\sum_{k=1}^{p^{\sz+1}} \mygcd(k,p^{\sz+1}) e^{-k\frac{2\pi  i}{p^{\sz+1}}m} =
   \totphi(p^{\sz+1}) 
\end{equation}

which completes the induction step.

\end{step} 


\end{case}

All cases being considered this finalizes the proof. 

\end{proof}

\begin{proof}[Proof 2, constructive proof]

The  proof   starts  from  a  result   \cite{Schramm2008},  that
expresses the  discrete Fourier transform  of a function  of the
greatest  common divisor  in the  Dirichlet convolution  of that
function and the Ramanujan sum \cite{Ramanujan1918}.

The Dirichlet  convolution, denoted $f  * g$, of  two arithmetic
functions, $f$ and $g$, is defined as

\begin{equation}\label{eq013}
f * g (n) = \sum_{d|n} f(\frac{n}{d}) g(d)
\end{equation}

and the Ramanujan sum as

\begin{equation}\label{eq014}
c_n(m) = \sum_{\substack{k=1\\\mygcd(k,n)=1}}^n e^{k\frac{2\pi  i}{n}m}.
\end{equation}

For  notational  reasons  we  use  $r_m(n)  =  c_n(m)$  for  the
Ramanujan  sum  if it  is  used  in combination  with  Dirichlet
convolution, but also  use the more common  notation $c_n(m)$ in
other cases. The result referred to above \cite{Schramm2008} can
now be written as

\begin{equation}\label{eq015}
\sum_{k=1}^n f(\mygcd(k,n)) e^{-k\frac{2\pi  i}{n}m} =  f * r_m (n)
\end{equation}

The  case $f(n)  = \id(n)$, reading 

\begin{equation}\label{eq017}
\sum_{k=1}^n \mygcd(k,n) e^{-k\frac{2\pi  i}{n}m} = \id * r_m (n),
\end{equation}

is what is needed here. Using von Sterneck's arithmetic function
\cite{vonSterneck1902},  the   Ramanujan  sum  can   be  written
as\cite{HardyWright1938}

\begin{equation}\label{eq018}
r_m(n) = \mu(\frac{n}{\mygcd(m, n)})\frac{\totphi(n)}{\totphi(\frac{n}{\mygcd(m,n)})},
\end{equation}

where $\mu(n)$ is the  Möbius function. With this representation
of the Ramanujan sum equation \ref{eq017} becomes

\begin{eqnarray}\label{eq019}
\sum_{k=1}^n \mygcd(k,n) e^{-k\frac{2\pi  i}{n}m}
& = & 
\sum_{d|n} \frac{n}{d} \mu(\frac{d}{\mygcd(m, d)})\frac{\totphi(d)}{\totphi(\frac{d}{\mygcd(m,d)})}.
\end{eqnarray}

As before, only the case where $n$  is a power of a prime number
$p$ needs  to be considered.  In that  case the divisors  of $n$
also  are powers  of $p$.  Hence, take  $n=p^s$, and  $d =  p^b,
b  =  0,\cdots,  s$.  Again two  cases  are  considered,  namely
$\mygcd(m,p^s) > 1$ and $\mygcd(m,p^s) = 1$


\begin{case}[$\mygcd(m,p^s) =  1$] Then for all  divisors $d$ of
$n$ it  holds that $\mygcd(m,d)  = 1$.  In the summation  on the
right  hand  side  of  equation \ref{eq019}  the  only  non-zero
contributions arise from the terms where $d = 1$, i.e., $b = 0$,
and $d  = p$, i.e.,  $b =  1$. For all  other values of  $b$ the
Möbius  function  gives zero.  Thus  one  obtains from  equation
\ref{eq019}, with $n = p^s$

\begin{eqnarray}\label{eq024}
\sum_{k=1}^n \mygcd(k,n) e^{-k\frac{2\pi  i}{n}m}
& = & 
p^s\frac{\totphi(1)}{\totphi(1)} + \frac{p^s}{p}(-1)\frac{\totphi(p)}{\totphi(p)}
\label{eq024-2}\\
& = & 
p^s - p^{s-1}\label{eq024-3}\\
& = & 
\totphi(p^s)\label{eq024-4}
\end{eqnarray}

which proves equation \ref{eq001}  of theorem \ref{th1} for this
case.

\end{case}

\begin{case}[$\mygcd(m,p^s) > 1$]

This means  that $m$ contains a  power of $p$, i.e.,  $m=u p^t$,
$t  \ge 1$,  and  $\mygcd(u,p) =  1$.  Equation \ref{eq019}  can
subsequently be written as

\begin{eqnarray}\label{eq020}
\sum_{k=1}^n \mygcd(k,n) e^{-k\frac{2\pi  i}{n}up^t}
& = & 
\sum_{b=0}^s \frac{p^s}{p^b} \mu(\frac{p^b}{\mygcd(up^t, p^b)})\frac{\totphi(p^b)}{\totphi(\frac{p^b}{\mygcd(up^t,p^b)})}.
\end{eqnarray}

The  totient function  has the  property $\totphi(p^b)  = p^b  -
p^{b-1},  b  \ge  1$,  whereas  $\totphi(p^0)  =  1$.  Therefore
the  term  $b=0$ will be  taken  out  of  the summation  in  equation
\ref{eq020}, the  value of this term  equals $p^s$, irrespective
of $t$.

Now  partition the  sum over  $b$ into  two subranges,  $b =  0,
\cdots, t$ and $b = t  +1, \cdots, s$, where the second subrange
is absent if $t \ge s$. Furthermore, use that $\mygcd(up^t, p^b)
= p^{\mymin{t}{b}}$, and that


\begin{equation}\label{eq022}
\mu(\frac{p^b}{\mygcd(up^t, p^b)}) = 
 \left\{
\begin{array}{lll}
\mu(1) & = & \quad\quad 1 \quad  b \le t  \\ 
\mu(p^{b-t}) & = &
\left\{
\begin{array}{rl}
-1 & b = t + 1\\
0 & b > t + 1
\end{array}
\right.
\end{array}
\right.
\end{equation}

then equation \ref{eq020} becomes

\begin{eqnarray}
&&\sum_{k=1}^n \mygcd(k,n) e^{-k\frac{2\pi  i}{n}m}  \nonumber\\ 
&=& p^s + \sum_{b=1}^t \frac{p^s}{p^b}\frac{\totphi(p^b)}{1} + \sum_{b=t+1}^s \frac{p^s}{p^b}
\mu(p^{b-t}) \frac{\totphi(p^b)}{\totphi(\frac{p^b}{\mygcd(up^t,p^b)})} \label{eq023-1} \\ 
&=&p^s + \sum_{b=1}^t \frac{p^s}{p^b}(p^b - p^{b-1}) - 
\frac{p^s}{p^{t+1}}
\frac{\totphi(p^{t+1})}{\totphi(\frac{p^{t+1}}{up^t})}(1-\mytheta{t}{s}) \label{eq023-2}  \\ 
&=&(\mymin{t}{s} + 1) \totphi(p^s) + \mytheta{t}{s} p^{s-1}\label{eq023-3}
\end{eqnarray}

Because $\mygcd(m,  p^s) >  1$ the  first summation  in equation
\ref{eq023-1} contains  at least one term,  corresponding to the
case $t=1$. In the second summation only  the term with $b = t +
1$  remains,  again because  of  the  properties of  the  Möbius
function, except in the case where  $t \ge s$, when this term is
not present  at all, resulting in  the factor $1-\mytheta{t}{s}$
in equation  \ref{eq023-2}. After cancelling a  factor $p^{b-1}$
in the  remaining summation on  the right hand side  of equation
\ref{eq023-2} the summand no longer  depends on $b$ resulting in
a factor  $\mymin{t}{s}$. The  $\min$ function results  from the
fact that the summation over $b$ runs up to $t$ except when $s <
t$, in which case it runs up to $b=s$. Equation \ref{eq023-3} is
the desired result.

\end{case}

\end{proof}

\section{Discussion}

In this section a number  of interesting consequences of theorem
\ref{th1} are  discussed, some  of which are  generalizations of
known results.

\begin{corollary}\label{corollary001}  If  $\mygcd(m,n)=1$  then
the discrete  Fourier transform of the  greatest common divisor,
as  given  in equation  \ref{eq000-1},  is  equal to  the  Euler
totient  function, i.e.,  if $\mygcd(m,n)  = 1$  then $h_m(n)  =
\totphi(n).$

\end{corollary}

\begin{proof}

From theorem  \ref{th1} it  follows that  if $\mygcd(m,n)  = 1$,
then  $t_i =  0$, implying  that $\mytheta{s_i}{t_i}  = 0$,  and
$\mymin{t_i}{s_i}  =  0$, $i  =  1,\cdots,  r$. Because  of  the
multiplicativity of the totient function one obtains

\begin{eqnarray}\label{eq025}
\sum_{k=1}^n \mygcd(k,n) e^{-k\frac{2\pi  i}{n}m}
& = &
  \prod_{i = 1}^r  \totphi(p_i^{s_i})\\
& = &
\totphi(n).
\end{eqnarray}

\end{proof}

The  result of  corollary \ref{corollary001}  was known  for the
special case $m = 1$, see \cite{Schramm2008}.

\begin{corollary}\label{corollary002}

If $\mygcd(m,n) = 1$ then the Ramanujan sum $r_m(n) = \mu(n)$

\end{corollary}

\begin{proof}

From corollary \ref{corollary001} and equation \ref{eq017} it follows that
if $\mygcd(m,n) = 1$ then 

\begin{eqnarray}\label{eq026}
id * r_m(n) & = & \totphi(n)
\end{eqnarray}

which  combined   with  the  well   known  result  that   $id  *
\mu(n)  = \totphi(n)$ proves corollary \ref{corollary002}.

\end{proof}

Corollaries \ref{corollary001} and \ref{corollary002} above both
apply to  the case where  $\mygcd(m,n) =  1$, in which  case the
function  $h_m(n)$ becomes  entirely  independent of  $m$. As  a
matter of fact,  also for general $m$ does  $h_m(n)$ only depend
on prime factors $p_i$ of $n$. The only specific property of $m$
that affects  the result is  the multiplicity $t_i$ of  $p_i$ in
$m$, as can be seen from theorem~\ref{th1}.

\section{Generalizations}

It is  possible to generalize  theorem \ref{th1} to the  case of
the  discrete Fourier  transform  \textit{of a  function of  the
greatest common  divisor}. Hence we define  the discrete Fourier
transform  of  a function  of  the  greatest common  divisor  as
follows.  Given an  arithmetic function  $f(n) :  \mathbb{N} \to
\mathbb{C}$,  the \textit{discrete  Fourier  transform} of  this
function of the \textit{greatest common divisor} with respect to
the $m$-th order of the $n$-th root of unity, $1 \le m\le n$, is
given by

\begin{equation}\label{eq000-0}
\hat{f}_m(n) = \sum_{k=1}^n f(\mygcd(k,n)) e^{-k\frac{2\pi  i}{n}m}
\end{equation}

As already made  clear in the introduction,  $\hat{f}_m(n)$ is a
multiplicative function for fixed $m$ if $f$ is a multiplicative
function. Moreover, if $f$ is integer valued then $\hat{f}_m(n)$
is integer  valued, which also  follows from~\cite{Schramm2008}.
Taking  the  the identity  function  $id(n)  := n$  for  $f(n)$,
equation \ref{eq000-0}  gives the discrete Fourier  transform of
the $\mygcd$, $h_m(n)$.

  A closed form expression  in the case
where $f$  is \textit{completely multiplicative} will  be proven
below.  The  following  lemma  covers  the  general  case  of  a
multiplicative function of the greatest common divisor, in which
case the result still contains a summation. When specializing to
a completely multiplicative function this summation can be done.

\begin{lemma}\label{lemma001}

If  the factorization  of the  number $n$  in its  prime factors
$p_j, j  = 1,  \cdots, r$,  is given as  $n =  \prod_{j=1}^r p_j
^{s_j}$, where $s_j$,  with $s_j \ge 1$, is  the multiplicity of
prime factor $p_j$,  and the number $m$, where $1  \le m \le n$,
is written  as $m = u  \prod_{i=1}^r p_i ^{t_i}$, with  $t_i \ge
0$, $\mygcd(u,p_i) = 1$, and $1  \le i \le r$, then the discrete
Fourier  transform  of  a  multiplicative function  $f$  of  the
greatest common divisor is given by

\begin{eqnarray}\label{eq1001}
&&\sum_{k=1}^n f(\mygcd(k,n)) e^{-k\frac{2\pi  i}{n}m}
 = \nonumber\\
&& \ \prod_{i = 1}^r \left[ 
f(p_i^{s_i}) + (p_i-1) \sum_{b=1}^{\mymin{t_i}{s_i}} p_i^{b-1}f(p_i^{s_i-b}) -
 f(p_i^{s_i-t_i-1})p_i^{t_i}(1-\mytheta{t_i}{s_i})
  \right]
\end{eqnarray}

\end{lemma}

\begin{proof}  The proof  is  similar to  the  proof of  theorem
\ref{th1}, we  repeat the  constructive proof based  on equation
\ref{eq015}  and von  Sterneck's formula,  equation \ref{eq018},
where now the identity function $\id(n)$ is replaced by $f(n)$,

\begin{eqnarray}\label{eq1002}
\sum_{k=1}^n \mygcd(k,n) e^{-k\frac{2\pi  i}{n}m}
& = & 
\sum_{d|n} f(\frac{n}{d}) \mu(\frac{d}{\mygcd(m, d)})\frac{\totphi(d)}{\totphi(\frac{d}{\mygcd(m,d)})}.
\end{eqnarray}

Steps   similar   to   those  from   equation   \ref{eq020}   to
\ref{eq023-3} for the  case $\mygcd(m,n) > 1$  and from equation
\ref{eq024} to equation \ref{eq024-3}  for the case $\mygcd(m,n)
= 1$ produce the result of lemma \ref{lemma001}.

\end{proof}

The  result  of  theorem   \ref{th1}  is  recovered  from  lemma
\ref{lemma001} by taking  $f = id$. In that case  the summand of
the summation  in equation  \ref{eq1001} becomes  independent of
the  summation  index which  means  that  the summation  results
in  a factor  $\mymin{t}{s}$.  After  regrouping terms  equation
\ref{eq001} of theorem \ref{th1} is obtained.

In  order  to  make  further  progress  it  is  needed  to  make
assumptions about the function $f$. We will consider the natural
case  of  specializing  to a  \textit{completely  multiplicative
function}. In  this case  the summation in  lemma \ref{lemma001}
can be explicitly performed.

\begin{theorem}\label{th2}  The  discrete Fourier  transform  of
a  completely  multiplicative function  $f$,  not  equal to  the
identity function $id$, of the greatest common divisor is, under
the conditions of lemma \ref{lemma001}, given by

\begin{eqnarray}\label{eq1003}
&&\sum_{k=1}^n f(\mygcd(k,n)) e^{-k\frac{2\pi  i}{n}m}
 = 
\prod_{i = 1}^r \bigg[ 
f(p_i^{s_i})  - f(p_i^{s_i-t_i-1})p_i^{t_i}(1-\mytheta{t_i}{s_i})   \nonumber\\ 
 &&  + (p_i-1) f(p_i^{s_i-1}) \frac{f(p^{\mymin{t_i}{s_i}})-p_i^{\mymin{t_i}{s_i}} }{f(p^{\mymin{t_i}{s_i}})-p_if(p_i^{\mymin{t_i}{s_i}-1})}(1 - \delta_{0t_i})
  \bigg]
\end{eqnarray}

\end{theorem}

\begin{proof} If  the function $f$ is  completely multiplicative
then $f(n^s) = f^s(n)$, and the series in equation \ref{eq1001},
lemma \ref{lemma001},  becomes geometric.  As a  consequence the
summation  can be  performed leading  to the  given closed  form
expression. The absence  of the summed term for the  case $t_i =
0$  has been  made  explicit with  factor  $1 -  \delta_{0t_i}$.
\end{proof}

If the function $f$ equals the identity function the closed form
formula  for a  geometric  series diverges.  However the  series
can  still be  summed  as  proven in  theorem  \ref{th1}. 

The sum function of an  arithmetic function $t(n)$ is defined as
the Dirichlet product  of $t(n)$ with the  constant one function
$1(n) := 1$ as

\begin{equation} 
S^t(n) = (1 * t) (n) \label{eq1026}
\end{equation}  

If  the function  $f$ is  written as  the sum  function of  some
function $t$, $f  = 1 * t$, then the  following corollary can be
formulated.

\begin{corollary}\label{corollary1002}

If  the function  $f(n)$ is  the  sum function  of the  function
$t(n)$, i.e., $f(n)  = 1 * t(n)$, then, for  those values of $m$
that  are such  that  $\mygcd(m,n) =  1$,  the discrete  Fourier
transform of  the sum  function of  the greatest  common divisor
gives back the function $t$, i.e., $\hat{f}_m(n) = t(n)$.

\end{corollary}

\begin{proof}

From  equations \ref{eq000-0}  and \ref{eq015}  it follows  that
$\hat{f}_m(n)  =  r_m  *  f(n)$,  which is  what  is  proven  in
\cite{Schramm2008}.  From corollary  \ref{corollary002} it  then
follows that $\hat{f}_m(n) = \mu(n) *  1 * t(n) = t(n) $ because
the Möbius function is the Dirichlet inverse of the constant one
function.

\end{proof}

As  an  interesting  example   of  the  application  of  theorem
\ref{th2} consider the power function

\begin{eqnarray}\label{eq10031}
\id_k(n)  = n^k. 
\end{eqnarray}

This is  a completely multiplicative  function not equal  to the
identity function if $k \ne 1$, hence theorem \ref{th2} applies.
On the other  hand, if $k = 1$, the  function $\id_k$ equals the
identity  function  and theorem  \ref{th1}  applies.  It can  be
verified  that  applying  theorem  \ref{th2} with  $f  =  \id_k$
and  then taking  the limit  $k\rightarrow 1$,  using l'Hopitals
rule, reproduces  the result  of theorem~\ref{th1}. We  will not
reproduce this here.


The  sum-function $S^f$  of  a \textit{multiplicative}  function
$f(n)$ can be expressed in  the prime-factors of the argument as
follows\cite{HardyWright1938}

\begin{eqnarray}\label{eq10002}
S^f(n) = \prod_{i = 1}^r \bigg[1 + f(p_i) + \cdots + f(p_i^{s_i})\bigg]
\end{eqnarray}

With the help of lemma \ref{lemma001} it is possible to obtain a
converse relation, as expressed by the following corollary.

\begin{corollary}\label{corollary1003} Given
$n=\prod_{i=1}^rp_i^{s_i}$, a multiplicative function $f$ can be
expressed in its sum-function $S^f$ as

\begin{eqnarray}\label{eq10003}
f(n) = \prod_{i = 1}^r \bigg[S^f(p_i^{s_i}) - S^f(p_i^{s_i - 1})\bigg]
\end{eqnarray}
\end{corollary}

\begin{proof} From lemma \ref{lemma001} and formula \ref{eq015},
taking $S^f$ for $f$, and  choosing $m$ such that $\mygcd(m,n) =
1$, it follows that

\begin{eqnarray}\label{eq10004}
S^f*r_m(n) = \prod_{i = 1}^r \bigg[S^f(p_i^{s_i}) - S^f(p_i^{s_i - 1})\bigg]
\end{eqnarray}

because the condition $\mygcd(m,n) = 1$ implies that $t_i = 0, i
= 1,\cdots,  r$, resulting in  both the uncompleted sum  and the
symbol $\mytheta{t_i}{s_i}$ in equation \ref{eq1001} being zero.
With  corollary  \ref{corollary002}  equation  \ref{eq10003}  of
corollary  \ref{corollary1003}  follows directly  from  equation
\ref{eq10004}.

\end{proof}

Note that  corollary \ref{corollary1003} applies  generally, and
has no dependence  on the order of unity  $m$ anymore. Corollary
\ref{corollary1003} reduces  to the familiar expression  for the
Euler totient function, valid for $n\ne 1$, 

\begin{eqnarray}\label{eq100041}
\totphi(n) = \prod_{i = 1}^r \bigg[p_i^{s_i} - p_i^{s_i - 1}\bigg]
\end{eqnarray}

if $\totphi$ is chosen for $f$ and one uses that the identity is
the  sum-function  of  the  totient-function,  as  expressed  by
the well known formula,

\begin{eqnarray}\label{eq10006}
\id * \mu (n) = \totphi(n)
\end{eqnarray}

Conversely, equation \ref{eq10004} can be written as

\begin{eqnarray}\label{eq10005}
f*\mu(n) = \prod_{i = 1}^r \bigg[f(p_i^{s_i}) - f(p_i^{s_i - 1})\bigg]
\end{eqnarray}

in which  form it can be  viewed as a generalization  of formula
\ref{eq10006}.  If the  power function  $\id_k$, see  expression
\ref{eq10031} above, is chosen for  the function $f$ in equation
\ref{eq10005}, the well known relation

\begin{eqnarray}\label{eq10007}
\id_k*\mu(n) = J_k(n)
\end{eqnarray}

for the Jordan function $J_k(n)$, defined as 

\begin{eqnarray}\label{eq10008}
J_k(n) = n^k \prod_{i=1}^r \left(1 - \frac{1}{p_i^k} \right)
\end{eqnarray}

is  obtained.  $J_k(n)$  is  a  generalization  of  the  totient
function with $\totphi(n) = J_1(n)$.

\section{The GCD sum function}

The  sum   of  greatest  common  divisors   is  called  Pillai's
arithmetical function\cite{Pillai1933}.  An expression  for this
sum exists  \cite{Broughan2001}, and  reads, in the  notation of
the referred work,

\begin{eqnarray}\label{eq10001}
\sum_{j=1}^{p^{\alpha}} \mygcd(j,p^{\alpha})
& = &
(\alpha+1)p^\alpha - \alpha p^{\alpha-1} \label{eq10001-1}\\
& = &
(\alpha+1)\totphi(p^\alpha) + p^{\alpha-1} \label{eq10001-2}
\end{eqnarray}

This  sum  is  just  a  special case  of  the  discrete  Fourier
transform of  the greatest  common divisor corresponding  to the
case  $m=n$.  By  virtue  of lemma  \ref{lemma001}  and  theorem
\ref{th2}  this  sum  can  now  be generalized  to  the  sum  of
\textit{a function} of the greatest common divisor. Hence we can
formulate the following corollary.

\begin{corollary}\label{corollary003}

Given a  multiplicative function $f$,  the sum  over $k, k  = 1,
\cdots  , n$  of  the function  values  $f(\mygcd(k,n))$ of  the
greatest common divisor can be expressed in the prime factors of
the number $n = \prod_{i=1}^r p_i ^{s_i}$ in the following way

\begin{eqnarray}\label{eq1005}
\sum_{k=1}^n f(\mygcd(k,n)) =
\prod_{i = 1}^r \bigg[ 
f(p_i^{s_i}) + (p_i-1) \sum_{b=1}^{s_i}p_i^{b-1} f(p_i^{s_i-b}) \bigg]
\end{eqnarray}

If  the  function  $f$  is completely  multiplicative,  but  not
equal  to the  identity  function $\id$,  the  sum in  expression
\ref{eq1005} can be evaluated, yielding

\begin{eqnarray}\label{eq1006}
\sum_{k=1}^n f(\mygcd(k,n)) =
\prod_{i = 1}^r \bigg[ 
f(p_i^{s_i}) 
  + (p_i-1) f(p_i^{s_i-1}) \frac{f(p^{s_i})-p_i^{s_i} }{f(p^{s_i})-p_if(p_i^{s_i-1})}
  \bigg]
\end{eqnarray}


\end{corollary}

\begin{proof}[Proof of  corollary \ref{corollary003}]  By taking
the  discrete   Fourier  transform  with  $m=n$   in  expression
\ref{eq1001},  it follows  that $t_i  =  s_i, i  = 1,\cdots,  r$
implying that $\mymin{t_i}{s_i}=s_i$ and $1 - \mytheta{s_i}{t_i}
=   0$   and   consequently    the   statements   of   corollary
\ref{corollary003} follow immediately from lemma \ref{lemma001},
and theorem \ref{th2}.\end{proof}


\begin{table}[bth]
\begin{equation*}
\begin{array}{lll}
k,m&\mygcd(k,n)&h_m(n)\\
\hline
1&1&\totphi(n)\\
\vdots&\vdots&\vdots\\
p-1&1&\totphi(n)\\
p&p&(2p-1)(q-1)\\
p+1&1&\totphi(n)\\
\vdots&\vdots&\vdots\\
2p-1&1&\totphi(n)\\
2p&p&(2p-1)(q-1)\\
2p+1&1&\totphi(n)\\
\vdots&\vdots&\vdots\\
q-1&1&\totphi(n)\\
q&q&(p-1)(2q-1)\\
q+1&1&\totphi(n)\\
\vdots&\vdots&\vdots\\
p^2-1&1&\totphi(n)\\
p^2&p&(2p-1)(q-1)\\
p^2+1&1&\totphi(n)\\
\vdots&\vdots&\vdots\\
n-1&1&\totphi(n)\\
n&n&(2p-1)(2q-1)
\end{array}
\end{equation*}
\caption{\label{table001}\footnotesize\bfseries Representation of $h_m(n), n = pq$}
\end{table}

\section{Example}

Consider  the  case  $n  =  pq$,  with  $p$  and  $q$  different
prime numbers,  and $p  < q$. The  totient function  then equals
$\totphi(pq)=(p-1)(q-1)$. View the discrete Fourier transform of
the  greatest  common divisor  as  a  mapping  of the  array  of
greatest  common  divisors,  indexed  by  $k$,  into  the  array
$h_m(n)$, defined in equation  \ref{eq000-1}, and indexed by the
order  of  the  root  of  unity  $m$.  Thus,  $k$  and  $m$  run
over  the  same set  of  values,  namely $1\cdots  n$.  Applying
theorem  \ref{th1}  results  in  the  representation,  given  in
table  \ref{table001},  of  the  function  $h_m(n)$  defined  in
equation~\ref{eq000-1}.

It  should be  kept in  mind that  the explicit  layout of  such
tables depends on the actual values  of the prime factors of the
number $n$. In the layout  of table \ref{table001}, for example,
it is assumed that  $p^2 > q$, but this could  of course just as
well be the other way round.

Note that there  are, in this case, only  four possible outcomes
for  $h_m(n)$.  It  is  a  general  property  of  $h_m(n)$  that
relatively  few values  are  possible. Consider  the case  where
$n=p^s$, then,  from theorem \ref{th1}, the  following cases for
$m$ can be distinguished.

\begin{equation}
h_m(n) = \begin{cases}
\hfill\totphi(p^s)&\quad(\mygcd(m,n) = 1)\\
\hfill (t + 1)\totphi(p^s)&\quad(\mygcd(m,n) = p^t, t < s)\\
\hfill (s + 1)\totphi(p^s) + p^{s-1}&\quad(\mygcd(m,n) = p^s)\\
\end{cases}
\end{equation}

Obviously,  tables  similar  to   table  \ref{table001}  can  be
constructed  for  all  $n$  given  the  prime  factorization  of
$n$.  As another  example, consider  the case  $n=p^3q^2w$, with
$p$,  $q$, and  $w$  different prime  numbers.  This results  in
the  representation given  in  table  \ref{table002}, this  time
including  just  one typical  entry.  The  totient function  now
equals $\totphi(n) = (p^3-p^2)(q^2-q)(w-1)$.

\begin{table}[th]
\begin{equation*}
\begin{array}{lll}
k,m&\mygcd(k,n)&h_m(n)\\
\hline
1&1&\totphi(n)\\
\vdots&\vdots&\vdots\\
p^2q^2&p^2q^2&3\totphi(p^2)\left[3\totphi(q^2)+q\right]\totphi(w)\\
\vdots&\vdots&\vdots\\
n&n&\left[4\totphi(p^3) + p^2\right]\left[3\totphi(q^2)+q\right] \left[2\totphi(w) + 1\right]\\
\end{array}
\end{equation*}
\caption{\label{table002}\footnotesize\bfseries Representation of $h_m(n), n = p^3q^2w$}
\end{table}

\section{Conclusion}

The discrete  Fourier transform  of a  function of  the greatest
common divisor has been expressed directly in function values of
the prime  factors of the  argument. Thus, an altenative  way of
calculating the transform is  established. Instead of evaluating
a  sum of  terms  containing exponential  factors, an  exprssion
based on  function values of powers  of prime factors has  to be
evaluated.  Naturally this  approach requires  knowledge of  the
prime factors.

Specializing  to the  identity-function, relationships  with the
Möbius  function  and  the  Euler  totient  function  have  been
obtained. Furthermore, by taking the order $m$ equal to the root
of  unity $n$,  an expression  in prime  factors of  the sum  of
function values of the greatest  common divisor could be stated.
Examples of explicit layouts, based  on the order $m$, have been
presented.


\section*{Acknowledgement} This research  has been made possible
through funding by the faculty of health, science and technology
at Karlstad university.

\section*{References}

\bibliography{mybib}

\begin{thebibliography}{1}
\expandafter\ifx\csname url\endcsname\relax
  \def\url#1{\texttt{#1}}\fi
\expandafter\ifx\csname urlprefix\endcsname\relax\def\urlprefix{URL }\fi
\expandafter\ifx\csname href\endcsname\relax
  \def\href#1#2{#2} \def\path#1{#1}\fi

\bibitem{Schramm2008}
W.~Schramm, The {F}ourier transform of functions of the greatest common
  divisor., Integers 8 (2008) A50.

\bibitem{vanderKamp2013}
P.~H. van~der Kamp, On the {F}ourier transform of the greatest common divisor,
  Integers 13 (2013) A24.

\bibitem{Kluyver1906}
J.~Kluyver, Some formulae concerning the integers less than $n$ and prime to
  $n$., KNAW, Proceedings, Amsterdam 9 I (1906) 408--414.

\bibitem{Apostol1976}
T.~M. Apostol, Introduction to analytic number theory, Springer-Verlag, Berlin
  Heidelberg, 1976.

\bibitem{Ramanujan1918}
S.~Ramanujan, On certain trigonometric sums and their applications in the
  theory of numb ers, Transactions of the Cambridge Philosophical Society 22
  (1918) 259--276.

\bibitem{vonSterneck1902}
R.~D. von Sterneck, Ein analogon zur additiven zahlentheorie, Sitzber, Akad.
  Wiss. Wien, Math. Naturw. Klasse I ll (Abt. I Ia) (1902) 1567--1601.

\bibitem{HardyWright1938}
G.~Hardy, E.~Wright, An Introduction to the Theory of Numbers (5th ed.),
  Clarendon Press, Oxford, 1979.

\bibitem{Pillai1933}
S.~Pillai, On an arithmetic function, Annamalai University Journal II (1933)
  242--248.

\bibitem{Broughan2001}
K.~A. Broughan, The gcd-sum function., Journal of integer sequences 4~(8)
  (2001) 1--16.

\end{thebibliography}

\end{document}